\documentclass[12pt,reqno]{amsart}

\usepackage{amssymb}
\usepackage{tikz-cd}
\usepackage{amsxtra}
\usepackage{verbatim}
\usepackage{enumerate}
\usepackage{enumitem}
\usepackage[english]{babel}
\usepackage{bm}
\usepackage{setspace}
\usepackage[x11names]{xcolor}
\usepackage[T1]{fontenc}
\usepackage{graphicx}
\usepackage{geometry}
\usepackage{amscd,latexsym,amsthm,amsmath,amsxtra}
\usepackage[colorlinks, linkcolor=DodgerBlue3, citecolor=DeepPink1]{hyperref}
\usepackage{enumerate}
\usepackage[all]{xy}
\usepackage{mathrsfs}

\providecommand{\U}[1]{\protect\rule{.1in}{.1in}}
\RequirePackage{amsmath}
\RequirePackage{amssymb}

\usepackage[OT2,T1]{fontenc}
\DeclareSymbolFont{cyrletters}{OT2}{wncyr}{m}{n}
\DeclareMathSymbol{\sha}{\mathalpha}{cyrletters}{"58}
\input cyracc.def

\usepackage{comment}

\newtheorem{thm}{Theorem}[section]

\newtheorem{lem}[thm]{Lemma}

\newtheorem{prop}[thm]{Proposition}
\theoremstyle{definition}

\theoremstyle{remark}
\newtheorem{exam}[thm]{Example}
\theoremstyle{remark}
\newtheorem{rem}[thm]{Remark}

\newtheorem*{theorem*}{Theorem}
\newtheorem*{proposition*}{Proposition}
\newtheorem*{lemma*}{Lemma}
\newtheorem*{corollary*}{Corollary}
\newtheorem*{question*}{Question}
\newtheorem*{conjecture*}{Conjecture}
\newtheorem*{claim*}{Claim}
\newtheorem*{introtheorem*}{Theorem}
\newtheorem*{introproposition*}{Proposition}
\newtheorem*{introlemma*}{Lemma}
\newtheorem*{introcorollary*}{Corollary}

\numberwithin{equation}{section}

\newcommand{\Pic}{\textup{Pic}}

\newcommand{\Gal}{\textup{Gal}}

\newcommand{\Hom}{\textup{Hom}}

\newcommand{\Spec}{\textup{Spec}}

\renewcommand{\P}{\mathbb{P}}

\newcommand{\Q}{\mathbb{Q}}

\newcommand{\Z}{\mathbb{Z}}

\newcommand{\Ad}{\mathbf{A}}

\newcommand{\ord}{\mathrm{ord}}

\DeclareMathOperator{\Cl}{Cl}
\newcommand{\Ok}{\mathcal{O}_k}

\DeclareMathOperator{\Gr}{Gr}
\DeclareMathOperator{\Frob}{Frob}

\DeclareMathOperator{\SL}{SL}
\DeclareMathOperator{\GL}{GL}

\newcommand{\Ov}{\mathcal{O}_v}

\begin{document}

\title{Ideal class maps on projective spaces and Grassmannians}

\author{Yufan Liu}

\address{Yufan Liu
\newline University of Science and Technology of China,
\newline School of Mathematical Sciences,
\newline 96 Jinzhai Road,
\newline  230026 Hefei, Anhui, China
 }

\email{liuyufan@mail.ustc.edu.cn}


\begin{abstract}
We study ideal class maps on projective spaces and Grassmannians over a number field. We prove weak approximation in every fixed ideal class fiber on positive-dimensional linear subspaces and on Grassmannians. We also introduce local defects of morphisms of projective spaces and use them to describe the ideal classes arising from the image of a linear subspace as a finite union of cosets in the class group.
\end{abstract}

\maketitle

\section{Introduction}

Let \(k\) be a number field, let \(\Ok\) be its ring of integers, and let
\(\Cl(k)\) be its ideal class group. For \(N\geq 1\), consider the map
\[
\Phi_N:\P^N(k)\longrightarrow\Cl(k),\qquad
[x_0:\cdots:x_N]\longmapsto
\left[x_0\Ok+\cdots+x_N\Ok\right].
\]
This is well defined, since scaling the homogeneous coordinates changes
the associated fractional ideal only by a principal factor.

The ideal class generated by the homogeneous coordinates of a projective
point appears naturally in height theory over number fields. Schanuel
studied heights on projective space in \cite{Schanuel1979}, counting
rational points one ideal class at a time according to the ideal generated
by their homogeneous coordinates. Our aim here is different: rather than
counting points of bounded height, we study approximation properties within
each fixed ideal class.

Although \(\Phi_N\) is not induced by a morphism of varieties, its fibers
are large from the viewpoint of approximation. Our first result concerns
the restriction of \(\Phi_N\) to linear subspaces.

\begin{thm}[Theorem~\ref{thm:linear-subspaces}]\label{thm1}
Let \(\Lambda\subset\P^N_k\) be a positive-dimensional \(k\)-linear
subspace. Then, for every \(c\in\Cl(k)\), the set $\Phi_N^{-1}(c)\cap \Lambda(k)$ satisfies weak approximation on \(\Lambda\).
\end{thm}

In particular, every ideal class occurs on every positive-dimensional
linear subspace. The proof uses the saturated lattice
\(
V\cap\Ok^{N+1}
\)
associated with a linear subspace \(\Lambda=\P(V)\), together with strong
approximation for a special linear group.

\medskip

We next consider the Grassmannian \(\Gr(r,N)\) parametrizing
\(r\)-dimensional \(k\)-linear subspaces of \(k^N\), where
\(1\leq r<N\). The Pl\"ucker embedding is
\[
\iota:\Gr(r,N)\longrightarrow \P\bigl(\textstyle\bigwedge^r k^N\bigr),
\qquad
W\longmapsto [w_1\wedge\cdots\wedge w_r],
\]
where \(w_1,\ldots,w_r\) is any basis of \(W\). Composing \(\iota\) with
the ideal class map on the ambient projective space gives
\[
\Psi_{r,N}:=\Phi_{\binom Nr-1}\circ\iota:
\Gr(r,N)(k)\longrightarrow\Cl(k).
\]
Equivalently, \(\Psi_{r,N}(W)\) is the ideal class generated by the
Pl\"ucker coordinates of \(W\). 

\medskip

\begin{thm}[Theorem~\ref{thm:grassmannian}]\label{thm2}
For every \(c\in\Cl(k)\), the fiber
\(\Psi_{r,N}^{-1}(c)\) satisfies weak approximation on \(\Gr(r,N)\).
\end{thm}
\medskip

Finally, we study the behavior of ideal classes under morphisms of
projective spaces. Let
\(
f:\P^n_k\longrightarrow\P^m_k
\)
be a morphism of degree \(d\geq1\), and let
\(\Lambda\subset\P^n_k\) be a positive-dimensional linear subspace.
For each finite place \(v\) of \(k\), we define a locally constant integer-valued function measuring the local defect of \(f\) at \(v\).
The possible values of these functions form a subset
\(
\mathscr{D}_{f,\Lambda}\subseteq\Cl(k)
\),
defined in Section~\ref{sec:local-defects}, which completely determines the ideal class image of \(f(\Lambda(k))\).

\medskip

\begin{thm}[Theorem~\ref{thm:morphism-image}]\label{thm3}
With the notation above, we have
\[
\Phi_m\bigl(f(\Lambda(k))\bigr)
=
\mathscr{D}_{f,\Lambda}\cdot\Cl(k)^d.
\]
In particular, the ideal class image of \(f(\Lambda(k))\) is a finite
union of cosets of \(\Cl(k)^d\).
\end{thm}

\medskip

When \(\Lambda=\P^n_k\) and \(f\) admits an integral extension, we prove
that \(\Phi_m(f(\P^n(k)))\) is a single coset of \(\Cl(k)^d\).

\medskip    

The paper is organized as follows. In
Section~\ref{sec:preliminaries}, we fix notation and recall the
preliminary results used throughout the paper. In
Section~\ref{sec:linear-subspaces}, we study the ideal class map on
positive-dimensional linear subspaces of projective space and prove
Theorem~\ref{thm1}. In Section~\ref{sec:grassmannian}, we define the
ideal class map on Grassmannians via the Pl\"ucker embedding and prove
Theorem~\ref{thm2}. In Section~\ref{sec:local-defects}, we introduce the
local defects of a morphism of projective spaces and use them to prove
Theorem~\ref{thm3}. We also discuss morphisms admitting an integral
extension.

\section{Notation and preliminaries}
\label{sec:preliminaries}

Throughout the paper, \(k\) denotes a number field, \(\Ok\) its ring of
integers, and \(\Cl(k)\) its ideal class group. We identify \(\Cl(k)\)
with the group of isomorphism classes of invertible \(\Ok\)-modules. If
\(L\) is an invertible \(\Ok\)-module, we denote its class by \([L]\)
and put \(L^{-1}:=\Hom_{\Ok}(L,\Ok)\). For a positive integer \(d\), we
write \(\Cl(k)^d:=\{c^d:c\in\Cl(k)\}\) for the subgroup of \(d\)-th
powers in \(\Cl(k)\).

Let \(\Omega_k\) be the set of places of \(k\).
Let \(\Omega_k^f\) and \(\Omega_k^\infty\) denote the subsets of finite and infinite places, respectively.
For \(v\in\Omega_k\), let \(k_v\) denote the completion of \(k\) at \(v\).
If \(v\in\Omega_k^f\), let \(\mathfrak p_v\subset\Ok\) be the corresponding prime ideal.
Let \(\mathcal O_{k,v}:=(\Ok)_{\mathfrak p_v}\), and let \(\Ov\) denote the completion of \(\mathcal O_{k,v}\).
Equivalently, \(\Ov\) is the ring of integers of \(k_v\).
Let
\[
\ord_v:k_v^\times\longrightarrow\Z
\]
be the normalized valuation.

For a finite-dimensional \(k\)-vector space \(V\), we write
\(V_v:=V\otimes_k k_v\). If \(M\) is a finitely generated projective
\(\Ok\)-module, we write \(M_v:=M\otimes_{\Ok}\Ov\) for
\(v\in\Omega_k^f\).

For a finite set \(S\subset\Omega_k\), we denote by \(\Ad_k^S\) the adeles of \(k\) away from \(S\), namely
\[
\Ad_k^S
=
\prod_{v\in\Omega_k^\infty\setminus S} k_v
\times
\sideset{}{'}\prod_{v\in\Omega_k^f\setminus S} k_v .
\]
Here the restricted product over the finite places is taken with respect to the subrings \(\Ov\subset k_v\).
If \(S=\{v_0\}\), we write \(\Ad_k^{v_0}\) instead of \(\Ad_k^{\{v_0\}}\).

Let \(R\) be a domain and let \(M\) be a torsion-free \(R\)-module.
A submodule \(L\subset M\) is called saturated if \(M/L\) is torsion-free.
If \(R\) is a principal ideal domain and \(M\) is a finite free \(R\)-module, a nonzero vector \(x\in M\) is called primitive if \(Rx\subset M\) is saturated.
Equivalently, \(x\) can be extended to an \(R\)-basis of \(M\).
In particular, if \(R\) is a discrete valuation ring with maximal ideal \(\mathfrak m_R\), then \(x\in M\) is primitive if and only if \(x\notin \mathfrak m_RM\).

If \(M\) is a finitely generated projective \(\Ok\)-module of rank
\(r\), we write
\[
\det M:=\bigwedge^r M.
\]
Thus \(\det M\) is an invertible \(\Ok\)-module.

We recall the following form of the Steinitz classification.

\begin{prop}[Steinitz classification, {\cite[Theorem~1.2.19]{Cohen2000}}]\label{prop:steinitz}
Let \(M\) be a finitely generated projective \(\Ok\)-module of rank
\(r\geq1\). Then there is an invertible fractional ideal
\(\mathfrak a\) such that
\[
M\simeq\Ok^{r-1}\oplus\mathfrak a.
\]
Moreover, \([\mathfrak a]=[\det M]\). In particular, two finitely
generated projective \(\Ok\)-modules are isomorphic if and only if they
have the same rank and isomorphic determinants.
\end{prop}

We shall also use the following standard consequence of the Chebotarev
density theorem.

\begin{lem}\label{lem:prime-representative}
Let \(c\in\Cl(k)\), and let \(T\subset\Omega_k^f\) be a finite set.
Then there exists \(v\in\Omega_k^f\setminus T\) such that
\(
[\mathfrak p_v]=c
\).

In particular, every ideal class contains infinitely many prime ideals.
\end{lem}

\begin{proof}
Let \(H\) be the Hilbert class field of \(k\).
The Artin reciprocity theorem in \cite[Chapter~VI, Theorem~7.1]{Neukirch1999}, applied to the Hilbert class field, gives an isomorphism
\[
\Cl(k)\simeq\Gal(H/k)
\]
via the Artin symbol.
Let \(\sigma\in\Gal(H/k)\) be the element corresponding to \(c\) under this isomorphism.
The Chebotarev density theorem in \cite[Chapter~VII, Theorem~13.4]{Neukirch1999} shows that there are infinitely many finite places \(v\) of \(k\), unramified in \(H\), with \(\Frob_v=\sigma\).
Discarding the finitely many places in \(T\), we may choose such a \(v\notin T\).
For this \(v\), the defining property of the Artin symbol gives \([\mathfrak p_v]=c\).
\end{proof}

We shall also use strong approximation for special linear groups.

\begin{prop}[Strong approximation for \(\SL\), {\cite[Theorem~7.12]{PlatonovRapinchuk1994}}]\label{prop:strong-approximation}
Let \(V\) be a \(k\)-vector space of dimension at least \(2\).
Let \(v_0\in\Omega_k^f\).
Then
\[
\overline{\SL(V)(k)}=\SL(V)(\Ad_k^{v_0}),
\]
where the closure is taken in \(\SL(V)(\Ad_k^{v_0})\).
\end{prop}

We finish this section by recording a geometric interpretation of the ideal class map.

\begin{lem}\label{lem:phi-pullback}
Let \(P\in\P^N(k)\), and let \(\widetilde P:\Spec\Ok\to\P^N_{\Ok}\) be its unique extension.
Then we have
\[
\Phi_N(P)
=
\bigl[\widetilde P^*\mathcal O_{\P^N_{\Ok}}(1)\bigr]\in \Cl(k).
\]
\end{lem}

\begin{proof}
Choose homogeneous coordinates \(x=(x_0,\ldots,x_N)\in k^{N+1}\setminus\{0\}\) for \(P\), and set
\[
\mathfrak c_x:=x_0\Ok+\cdots+x_N\Ok\subset k.
\]
This is a nonzero finitely generated \(\Ok\)-submodule of \(k\), hence an invertible \(\Ok\)-module because \(\Ok\) is a Dedekind domain.
By definition, \(\Phi_N(P)=[\mathfrak c_x]\).
Since the \(x_i\) generate \(\mathfrak c_x\), the map
\[
\rho_x:\Ok^{N+1}\twoheadrightarrow\mathfrak c_x,
\qquad
e_i\longmapsto x_i,
\]
is surjective.

Let \(\mathcal O(1)=\mathcal O_{\P^N_{\Ok}}(1)\), with its standard generating sections \(X_0,\ldots,X_N\).
By the universal property of projective space, the pair \((\mathfrak c_x,\rho_x)\) defines a morphism \(s:\Spec\Ok\to\P^N_{\Ok}\) over \(\Ok\) such that \(s^*\mathcal O(1)\simeq\mathfrak c_x\) and \(s^*X_i=x_i\).
After tensoring \(\rho_x\) with \(k\), the corresponding \(k\)-point is \([x_0:\cdots:x_N]=P\).
Thus \(s=\widetilde P\).
Therefore \(\widetilde P^*\mathcal O(1)\simeq\mathfrak c_x\), and so
\[
\bigl[\widetilde P^*\mathcal O(1)\bigr]
=
[\mathfrak c_x]
=
\Phi_N(P).
\]
\end{proof}

\section{Fixed ideal classes on linear subspaces}
\label{sec:linear-subspaces}

In this section, we study the restriction of the ideal class map
\(\Phi_N\) to a positive-dimensional \(k\)-linear subspace
\(\Lambda\subset\P^N_k\). We describe the ideal class of a point in terms
of the associated local lattices and then use strong approximation for a
special linear group to prove weak approximation in every fiber of
\(\Phi_N|_{\Lambda(k)}\).

Let \(\Lambda\subset\P^N_k\) be a positive-dimensional \(k\)-linear
subspace. Write \(\Lambda=\P(V)\), where \(V\subset k^{N+1}\) is a
\(k\)-vector subspace of dimension \(r\geq2\), and set
\[
M:=V\cap\Ok^{N+1}.
\]

\begin{lem}\label{lem:associated-lattice}
The module \(M\) is a finitely generated projective \(\Ok\)-module of
rank \(r\), and \(M\otimes_{\Ok}k=V\). Moreover, for every
\(v\in\Omega_k^f\), we have
\[
M_v=V_v\cap\Ov^{N+1}
\]
inside \(k_v^{N+1}\).
\end{lem}

\begin{proof}
The module \(M\) is finitely generated because it is a submodule of the
finitely generated \(\Ok\)-module \(\Ok^{N+1}\). It is torsion-free and
hence projective, since \(\Ok\) is a Dedekind domain. Clearing
denominators in a \(k\)-basis of \(V\) shows that
\(M\otimes_{\Ok}k=V\), so \(M\) has rank \(r\).

By definition, there is an exact sequence
\[
0\longrightarrow M\longrightarrow\Ok^{N+1}
\longrightarrow k^{N+1}/V.
\]
By \cite[Proposition~3.3 and Proposition~10.14]{AtiyahMacdonald1969}, the ring \(\Ov\) is flat over \(\Ok\).
Tensoring with \(\Ov\) gives an exact sequence
\[
0\longrightarrow M_v\longrightarrow \Ov^{N+1}\longrightarrow (k^{N+1}/V)\otimes_{\Ok}\Ov .
\]
Since \(k\otimes_{\Ok}\Ov\simeq k_v\), we have
\[
(k^{N+1}/V)\otimes_{\Ok}\Ov
\simeq
(k^{N+1}/V)\otimes_k k_v
\simeq
k_v^{N+1}/V_v.
\]
Therefore \(M_v\) is the kernel of the natural map
\(\Ov^{N+1}\to k_v^{N+1}/V_v\), and hence
\[
M_v=V_v\cap\Ov^{N+1}.
\]
\end{proof}

Since \(r\geq2\), Proposition~\ref{prop:steinitz} allows us to fix a
decomposition
\(
M=\Ok e\oplus M'
\)
for some \(e\in M\) and some projective \(\Ok\)-module \(M'\). For
every \(v\in\Omega_k^f\), this induces a decomposition
\(M_v=\Ov e\oplus M'_v\). In particular, \(e\) is primitive in \(M_v\).

For \(v\in\Omega_k^f\) and
\(x=(x_0,\ldots,x_N)\in V_v\setminus\{0\}\), define
\[
\nu_v(x):=\min_{0\leq i\leq N}\ord_v(x_i).
\]
It follows from Lemma~\ref{lem:associated-lattice} that \(x\in M_v\) if
and only if \(\nu_v(x)\geq0\), and that a vector \(x\in M_v\) is
primitive if and only if \(\nu_v(x)=0\).

For every \(x=(x_0,\ldots,x_N)\in V\setminus\{0\}\), the fractional ideal
\(x_0\Ok+\cdots+x_N\Ok\) equals \(\prod_{v\in\Omega_k^f}\mathfrak p_v^{\nu_v(x)}\),
and therefore
\[
\Phi_N([x])
=
\left[
\prod_{v\in\Omega_k^f}\mathfrak p_v^{\nu_v(x)}
\right].
\]

For \(v\in\Omega_k^f\), set
\[
\SL(M_v):=\{g\in\SL(V)(k_v):g(M_v)=M_v\}.
\]
This is a compact open subgroup of \(\SL(V)(k_v)\).
For every finite set \(S\subset\Omega_k\), the group \(\SL(V)(\Ad_k^S)\) may be viewed as the restricted product of the groups \(\SL(V)(k_v)\), for \(v\notin S\), with respect to the subgroups \(\SL(M_v)\) at the finite places outside \(S\).

\begin{lem}\label{lem:primitive-transitivity}
For every \(v\in\Omega_k^f\), the group \(\SL(M_v)\) acts transitively
on the primitive vectors of \(M_v\).
\end{lem}

\begin{proof}
Let \(x\in M_v\) be primitive. Extend \(e\) and \(x\) to
\(\Ov\)-bases
\[
e,e_2,\ldots,e_r
\qquad\text{and}\qquad
x,x_2,\ldots,x_r
\]
of \(M_v\). Let \(h\in\GL(M_v)\) be the automorphism sending \(e\mapsto x\) and
\(e_i\mapsto x_i\) for \(2\leq i\leq r\), and put
\(u:=\det(h)\in\Ov^\times\). Replacing \(x_2\) by
\(u^{-1}x_2\), we obtain another \(\Ov\)-basis, and the corresponding
automorphism has determinant \(1\). Thus there is an element
\(g\in\SL(M_v)\) such that \(ge=x\).
\end{proof}

We now prove the main result of this section.

\begin{thm}\label{thm:linear-subspaces}
Let \(c\in\Cl(k)\) and \(S\subset\Omega_k\) be a finite set.
For each \(v\in S\), let \(\mathcal U_v\subset\Lambda(k_v)\) be a non-empty open subset.
Then there exists \(P\in\Lambda(k)\) such that $\Phi_N(P)=c$ and \(P\in\mathcal U_v\) for all \(v\in S\).

In other words, every fiber of \(\Phi_N|_{\Lambda(k)}\) satisfies weak approximation on \(\Lambda\).
\end{thm}

\begin{proof}
By Lemma~\ref{lem:prime-representative}, we may choose distinct finite places \(v_0,v_1\notin S\) such that $[\mathfrak p_{v_0}]=1$ and $[\mathfrak p_{v_1}]=c$.

Let \(G:=\SL(V)\).
For every \(v\in\Omega_k\), consider the map
\[
q_v:G(k_v)\longrightarrow\Lambda(k_v),
\qquad
g\longmapsto[ge].
\]
Since \(r\geq2\), the group \(G(k_v)\) acts transitively on \(V_v\setminus\{0\}\).
Hence \(q_v\) is surjective.

Let \(v\in S\cap\Omega_k^f\).
Choose \(P_v\in\mathcal U_v\) and a representative \(y_v\in V_v\setminus\{0\}\) of \(P_v\).
After multiplying \(y_v\) by an element of \(k_v^\times\), we may assume that \(y_v\in M_v\) is primitive.
By Lemma~\ref{lem:primitive-transitivity}, there exists \(g_v\in\SL(M_v)\) such that \(g_ve=y_v\).
Thus
\[
\mathcal W_v:=q_v^{-1}(\mathcal U_v)\cap\SL(M_v)
\]
is a non-empty open subset of \(G(k_v)\).
For every \(g\in\mathcal W_v\), the vector \(ge\) is primitive in \(M_v\).
Therefore \(\nu_v(ge)=0\).

Let \(v\in S\cap\Omega_k^\infty\).
Set
\[
\mathcal W_v:=q_v^{-1}(\mathcal U_v).
\]
Then \(\mathcal W_v\) is a non-empty open subset of \(G(k_v)\).

Since \(e\) is primitive in \(M_{v_1}\), it can be extended to an \(\mathcal{O}_{v_1}\)-basis \(e,e_2,\ldots,e_r\) of \(M_{v_1}\).
Choose a uniformizer \(\pi_{v_1}\in k_{v_1}\).
With respect to this basis, put
\[
d_{v_1}:=
\operatorname{diag}(\pi_{v_1},\pi_{v_1}^{-1},1,\ldots,1)
\in G(k_{v_1}).
\]
Then \(\nu_{v_1}(d_{v_1}e)=1\).
Since \(x\mapsto\nu_{v_1}(x)\) is locally constant on \(V_{v_1}\setminus\{0\}\), the set
\[
\mathcal W_{v_1}:=
\{g\in G(k_{v_1}):\nu_{v_1}(ge)=1\}
\]
is a non-empty open subset of \(G(k_{v_1})\).

For every \(v\in\Omega_k^f\setminus(S\cup\{v_0,v_1\})\), set \(\mathcal W_v:=\SL(M_v)\).
For every \(v\in\Omega_k^\infty\setminus S\), set \(\mathcal W_v:=G(k_v)\).
Then
\(
\mathcal W:=
\prod_{v\in\Omega_k\setminus\{v_0\}}\mathcal W_v
\)
is a non-empty open subset of \(G(\Ad_k^{v_0})\).

By Proposition~\ref{prop:strong-approximation}, there exists
\(
g\in G(k)\cap\mathcal W
\).
Set \(x:=ge\in V\setminus\{0\}\) and \(P:=[x]\in\Lambda(k)\).
By construction, we have \(P\in\mathcal U_v\) for every \(v\in S\).

Moreover, for every \(v\in\Omega_k^f\setminus\{v_0\}\), we have
\[
\nu_v(x)=
\begin{cases}
1, & v=v_1,\\
0, & v\neq v_1.
\end{cases}
\]
It follows that
\(
x_0\Ok+\cdots+x_N\Ok
=
\mathfrak p_{v_1}\mathfrak p_{v_0}^{\nu_{v_0}(x)}.
\)

Since \([\mathfrak p_{v_0}]=1\), we obtain
\(
\Phi_N(P)=[\mathfrak p_{v_1}]=c
\).
\end{proof}

\begin{rem}\label{rem:module-nonemptiness}
The non-emptiness assertion in Theorem~\ref{thm:linear-subspaces} also has a direct module-theoretic interpretation.
For a nonzero vector \(x\in V\), set
\[
L_x:=M\cap kx.
\]
Then the assignment \(kx\mapsto L_x\) gives a bijection between the \(k\)-lines in \(V\) and the saturated rank-one submodules of \(M\).
The inverse sends a saturated rank-one submodule \(L\subset M\) to the \(k\)-line \(L\otimes_{\Ok}k\subset V\).

If \(x=(x_0,\ldots,x_N)\) and
\(\mathfrak c_x:=x_0\Ok+\cdots+x_N\Ok\), then the map
\[
\mathfrak c_x^{-1}\longrightarrow L_x,
\qquad
a\longmapsto ax,
\]
is an isomorphism. 
Consequently, the ideal class \(\Phi_N([x])\) equals \([L_x]^{-1}\).

Conversely, let \(c\in\Cl(k)\), and choose an invertible
\(\Ok\)-module \(L\) with \([L]=c^{-1}\). Consider the projective module
\[
L\oplus\Ok^{r-2}\oplus
\bigl(\det M\otimes_{\Ok}L^{-1}\bigr).
\]
This module has the same rank and determinant as \(M\), and hence is
isomorphic to \(M\) by Proposition~\ref{prop:steinitz}. Therefore \(M\) contains a rank-one direct summand, hence a saturated rank-one submodule, isomorphic to \(L\). The corresponding
\(k\)-line
\[
L\otimes_{\Ok}k\subset M\otimes_{\Ok}k=V
\]
determines a point \(P\in\Lambda(k)\) satisfying
\(
\Phi_N(P)=[L]^{-1}=c
\).

This gives a direct algebraic proof that every fiber is nonempty, while
Theorem~\ref{thm:linear-subspaces} strengthens this observation to weak
approximation in every fixed ideal class.
\end{rem}

\section{Fixed ideal classes on Grassmannians}
\label{sec:grassmannian}

In this section, we define an ideal class map on Grassmannians using the
Pl\"ucker embedding and prove weak approximation in each of its fibers.

Let \(1\leq r<N\), and let \(\Gr(r,N)\) denote the Grassmannian
parametrizing \(r\)-dimensional \(k\)-linear subspaces of \(k^N\). The
Pl\"ucker embedding is
\[
\iota:\Gr(r,N)\longrightarrow
\P\left(\bigwedge^r k^N\right),
\qquad
W\longmapsto[w_1\wedge\cdots\wedge w_r],
\]
where \(w_1,\ldots,w_r\) is any \(k\)-basis of \(W\). We define
\[
\Psi_{r,N}:=
\Phi_{\binom Nr-1}\circ\iota:
\Gr(r,N)(k)\longrightarrow\Cl(k).
\]
Thus \(\Psi_{r,N}(W)\) is the ideal class generated by the Pl\"ucker
coordinates of \(W\).

We first give an intrinsic description of this ideal class. For
\(W\in\Gr(r,N)(k)\), set
\[
M_W:=W\cap\Ok^N.
\]
Applying Lemma~\ref{lem:associated-lattice} to the subspace \(W\subset k^N\), we see that \(M_W\) is a finitely generated projective \(\Ok\)-module of rank \(r\), and
\(
(M_W)_v=W_v\cap\Ov^N
\)
for every \(v\in\Omega_k^f\).

\begin{prop}\label{prop:grassmannian-determinant}
For every \(W\in\Gr(r,N)(k)\), we have
\(
\Psi_{r,N}(W)=[\det M_W]^{-1}
\).
\end{prop}

\begin{proof}
Let \(e_1,\ldots,e_N\) be the standard basis of \(k^N\), and let
\(e_I=e_{i_1}\wedge\cdots\wedge e_{i_r}\), where
\(I=\{i_1<\cdots<i_r\}\). Choose a \(k\)-basis
\(w_1,\ldots,w_r\) of \(W\), and write
\[
q:=w_1\wedge\cdots\wedge w_r
=
\sum_I q_Ie_I.
\]

Put $\mathfrak c(q):=\sum_I q_I\Ok$.
By definition, we have $\Psi_{r,N}(W)=[\mathfrak c(q)]$.

For \(v\in\Omega_k^f\), choose an \(\Ov\)-basis \(m_{v,1},\ldots,m_{v,r}\) of \((M_W)_v\), and put
\(
m_v:=m_{v,1}\wedge\cdots\wedge m_{v,r}
\).
Since \((M_W)_v\) is a direct summand of \(\Ov^N\), the vector \(m_v\) is primitive in \(\bigwedge^r\Ov^N\).
As \(q\) and \(m_v\) span the same one-dimensional \(k_v\)-vector space, there is \(\lambda_v\in k_v^\times\) such that \(q=\lambda_vm_v\).
It follows that
\(
\ord_v\bigl(\mathfrak c(q)\bigr)=\ord_v(\lambda_v)
\).

On the other hand, we regard \(\det M_W=\bigwedge^rM_W\) as an \(\Ok\)-lattice in the one-dimensional \(k\)-vector space \(\bigwedge^rW\).
Then
\[
(\det M_W)_v
=
\bigwedge^r(M_W)_v
=
\Ov m_v
=
\lambda_v^{-1}\Ov q.
\]

Since \(\ord_v(\mathfrak c(q))=\ord_v(\lambda_v)\) for every \(v\in\Omega_k^f\), these local descriptions give
\(
\det M_W=\mathfrak c(q)^{-1}\Ok q
\)
as lattices in \(\bigwedge^rW\).
Therefore we have
\(
[\det M_W]=[\mathfrak c(q)]^{-1}
\).

Since \(\Psi_{r,N}(W)=[\mathfrak c(q)]\), the assertion follows.
\end{proof}

The preceding proposition also gives a direct proof that every ideal
class occurs.

\begin{lem}\label{lem:grassmannian-nonempty}
For every \(c\in\Cl(k)\), there exists
\(W_0\in\Gr(r,N)(k)\) such that \(\Psi_{r,N}(W_0)=c\).
\end{lem}

\begin{proof}
Choose an invertible \(\Ok\)-module \(L\) with \([L]=c^{-1}\), and set
\[
M_0:=\Ok^{r-1}\oplus L,
\qquad
M_0':=\Ok^{N-r-1}\oplus L^{-1}.
\]
Then \(M_0\oplus M_0'\) has rank \(N\) and trivial determinant. By
Proposition~\ref{prop:steinitz}, we have
\[
M_0\oplus M_0'\simeq\Ok^N.
\]
Fix such an isomorphism and regard \(M_0\) as a direct summand of
\(\Ok^N\). Put
\[
W_0:=M_0\otimes_{\Ok}k\subset k^N.
\]
Since \(M_0\) is saturated in \(\Ok^N\), we have
\(M_{W_0}=M_0\). Proposition~\ref{prop:grassmannian-determinant} now
gives
\[
\Psi_{r,N}(W_0)
=
[\det M_0]^{-1}
=
[L]^{-1}
=
c.
\]
\end{proof}

We shall use the following local transitivity property.

\begin{lem}\label{lem:grassmannian-local-transitivity}
For every \(v\in\Omega_k^f\), the group \(\SL_N(\Ov)\) acts
transitively on \(\Gr(r,N)(k_v)\).
\end{lem}

\begin{proof}
Let \(W,W'\in\Gr(r,N)(k_v)\).
Then $M_W:=W\cap\Ov^N$ and $M_{W'}:=W'\cap\Ov^N$ are rank-\(r\) direct summands of \(\Ov^N\).
Choose \(\Ov\)-bases of \(M_W\) and \(M_{W'}\), and extend them to \(\Ov\)-bases of \(\Ov^N\).
Let \(h\in\GL_N(\Ov)\) be the corresponding change-of-basis matrix.
Then \(hW=W'\).

Put \(u:=\det(h)\in\Ov^\times\).
Multiplying one vector in the target basis by \(u^{-1}\), we obtain another \(\Ov\)-basis of \(\Ov^N\), whose first \(r\) vectors still span \(M_{W'}\).
The corresponding change-of-basis matrix has determinant \(1\) and still sends \(W\) to \(W'\).
Thus there exists \(g\in\SL_N(\Ov)\) such that \(gW=W'\).
\end{proof}

\medskip
We now prove weak approximation in every fixed ideal class fiber on the Grassmannian.

\begin{thm}\label{thm:grassmannian}
Let \(c\in\Cl(k)\), let \(S\subset\Omega_k\) be a finite set, and let
\(\mathcal{U}_v\subset\Gr(r,N)(k_v)\) be a non-empty open subset for
each \(v\in S\). Then there exists \(W\in\Gr(r,N)(k)\) such that $\Psi_{r,N}(W)=c$ and $W\in\mathcal{U}_v$ for all \(v\in S\).

In other words, every fiber of \(\Psi_{r,N}\) satisfies weak
approximation on \(\Gr(r,N)\).
\end{thm}

\begin{proof}
By Lemma~\ref{lem:grassmannian-nonempty}, we can choose
\(W_0\in\Gr(r,N)(k)\) such that
\(
\Psi_{r,N}(W_0)=c
\).
By Lemma~\ref{lem:prime-representative}, we can choose
\(v_0\in\Omega_k^f\setminus S\) such that
\(
[\mathfrak p_{v_0}]=1
\).

Let \(G:=\SL_N\). For every \(v\in\Omega_k\), consider the map
\[
q_v:G(k_v)\longrightarrow\Gr(r,N)(k_v),
\qquad
g\longmapsto gW_0.
\]

Let \(v\in S\cap\Omega_k^f\). By
Lemma~\ref{lem:grassmannian-local-transitivity}, there exists
\(g_v\in\SL_N(\Ov)\) such that \(g_vW_0\in\mathcal{U}_v\). Hence
\(
\mathcal{W}_v:=
q_v^{-1}(\mathcal{U}_v)\cap\SL_N(\Ov)
\)
is a non-empty open subset of \(G(k_v)\).

For \(v\in S\cap\Omega_k^\infty\), set
\(
\mathcal{W}_v:=q_v^{-1}(\mathcal{U}_v)
\).
Since \(G(k_v)\) acts transitively on \(\Gr(r,N)(k_v)\), this is a
non-empty open subset of \(G(k_v)\).

For every
\(v\in\Omega_k^f\setminus(S\cup\{v_0\})\), set
\(\mathcal{W}_v:=\SL_N(\Ov)\), and for every
\(v\in\Omega_k^\infty\setminus S\), set
\(\mathcal{W}_v:=G(k_v)\).  
Then
\(
\mathcal{W}:=
\prod_{v\in\Omega_k\setminus\{v_0\}}\mathcal{W}_v
\)
is a non-empty open subset of \(G(\Ad_k^{v_0})\).

By Proposition~\ref{prop:strong-approximation}, there exists
\(
g\in G(k)\cap\mathcal{W}
\).
Put \(W:=gW_0\). 
By construction, we have $W \in\mathcal{U}_v$ for every \(v\in S\).

It remains to prove that \(\Psi_{r,N}(W)=c\).
Choose a nonzero vector \(q_0\in\bigwedge^r W_0\), and put \(q:=(\bigwedge^r g)q_0\).
Then \(q\) is a system of Pl\"ucker coordinates for \(W\).
For every finite place \(v\neq v_0\), the construction gives \(g\in\SL_N(\Ov)\), and hence \(\bigwedge^r g\in\GL(\bigwedge^r\Ov^N)\).
Therefore the fractional ideals generated by the coordinates of \(q\) and \(q_0\) have the same valuation at every finite place \(v\neq v_0\).
Since \([\mathfrak p_{v_0}]=1\), the two fractional ideals have the same ideal class.
We conclude that \(\Psi_{r,N}(W)=\Psi_{r,N}(W_0)=c\).
\end{proof}

\section{Local defects of morphisms of projective spaces}
\label{sec:local-defects}

In this section, we introduce local defect functions associated with a
morphism of projective spaces and use them to determine its ideal class
image on every positive-dimensional linear subspace.

Let \(f:\P^n_k\to\P^m_k\) be a morphism of degree \(d\geq1\), so that \(f^*\mathcal O_{\P^m_k}(1)\simeq\mathcal O_{\P^n_k}(d)\).
Choosing such an isomorphism, we may write \(f\) as
\[
f([x_0:\cdots:x_n])
=
[F_0(x):\cdots:F_m(x)],
\]
where \(F_0,\ldots,F_m\in k[X_0,\ldots,X_n]\) are homogeneous polynomials of degree \(d\) with no common nontrivial zero over \(\overline{k}\).
We write \(F=(F_0,\ldots,F_m)\).

For every \(v\in\Omega_k^f\) and every
\(P=[x_0:\cdots:x_n]\in\P^n(k_v)\), define
\[
\delta_{F,v}(P)
:=
\min_{0\leq i\leq m}\ord_v(F_i(x))
-
d\min_{0\leq j\leq n}\ord_v(x_j).
\]

\begin{lem}\label{lem:local-defect-properties}
For every \(v\in\Omega_k^f\), the function
\[
\delta_{F,v}:\P^n(k_v)\longrightarrow\Z
\]
is well defined and locally constant. Its image is finite, and
\(\delta_{F,v}=0\) for all but finitely many \(v\in\Omega_k^f\).
\end{lem}

\begin{proof}
If \(x\) is replaced by \(\lambda x\), where \(\lambda\in k_v^\times\), then \(F_i(\lambda x)=\lambda^dF_i(x)\).
Both terms in the definition of \(\delta_{F,v}\) therefore increase by \(d\cdot\ord_v(\lambda)\), so \(\delta_{F,v}\) is well defined.

The function \((y_0,\ldots,y_s)\mapsto\min_i\ord_v(y_i)\) is locally constant on \(k_v^{s+1}\setminus\{0\}\).
Since the polynomials \(F_0,\ldots,F_m\) have no common nontrivial zero over \(\overline{k}\), the tuple \(F(x)\) is nonzero for every \(x\in k_v^{n+1}\setminus\{0\}\).
It follows that \(\delta_{F,v}\) is locally constant on \(\P^n(k_v)\).
Since \(\P^n(k_v)\) is compact, the image of \(\delta_{F,v}\) is finite.

It remains to prove the final assertion.
By the homogeneous Nullstellensatz, the ideal \((F_0,\ldots,F_m)\) contains a power of the irrelevant ideal \((X_0,\ldots,X_n)\).
After excluding finitely many finite places, all \(F_i\) have coefficients in \(\mathcal O_{k,v}\), and the containment remains valid over \(\mathcal O_{k,v}\).
Hence the reductions \(\overline{F_0},\ldots,\overline{F_m}\) have no common nontrivial zero over \(\Ok/\mathfrak p_v\).

Let \(v\) be such a place.
Represent \(P\in\P^n(k_v)\) by a primitive vector \(x=(x_0,\ldots,x_n)\in\Ov^{n+1}\).
Then the reduction \(\bar x\) is a nonzero vector over \(\Ok/\mathfrak p_v\).
If \(F_i(x)\in\mathfrak p_v\Ov\) for every \(i\), then all reductions \(\overline{F_i}\) vanish at \(\bar x\), a contradiction.
Thus at least one \(F_i(x)\) is a unit.
Since all \(F_i(x)\) lie in \(\Ov\), we obtain \(\min_i\ord_v(F_i(x))=0\).
Since \(x\) is primitive, we have \(\min_j\ord_v(x_j)=0\), and hence \(\delta_{F,v}(P)=0\).
\end{proof}

Let \(\Lambda\subset\P^n_k\) be a positive-dimensional \(k\)-linear
subspace. For every \(v\in\Omega_k^f\), put
\[
D_{F,v,\Lambda}:=
\delta_{F,v}\bigl(\Lambda(k_v)\bigr)\subset\Z.
\]
By Lemma~\ref{lem:local-defect-properties}, each
\(D_{F,v,\Lambda}\) is finite and
\(D_{F,v,\Lambda}=\{0\}\) for all but finitely many \(v\).

Define
\[
\mathscr{D}_{f,\Lambda}
:=
\left\{
\prod_{v\in\Omega_k^f}
[\mathfrak p_v]^{a_v}
\;:\;
a_v\in D_{F,v,\Lambda}
\text{ for every }v
\right\}
\subset\Cl(k).
\]
The products in this definition are finite, and
\(\mathscr{D}_{f,\Lambda}\) is a finite subset of \(\Cl(k)\).

\begin{lem}\label{lem:defect-independent}
The subset \(\mathscr{D}_{f,\Lambda}\) is independent of the choice of
the homogeneous tuple \(F\) representing \(f\).
\end{lem}

\begin{proof}
Any other homogeneous tuple \(F'=(F'_0,\ldots,F'_m)\) of degree \(d\)
representing \(f\) satisfies \(F'_iF_j=F'_jF_i\) for all \(i,j\) in
\(k[X_0,\ldots,X_n]\). Since \(n\geq1\) and neither tuple has a common
nontrivial zero over \(\overline{k}\), the entries of each tuple have no
common factor of positive degree in this factorial ring. A standard
greatest-common-divisor argument then shows
\(
F'=\lambda F
\)
for some \(\lambda\in k^\times\). 

For every
\(v\in\Omega_k^f\), we have
\(
\delta_{F',v}
=
\delta_{F,v}+\ord_v(\lambda),
\)
and hence
\(
D_{F',v,\Lambda}
=
D_{F,v,\Lambda}+\ord_v(\lambda)
\).
It follows that the subset defined using \(F'\) is obtained from the
subset defined using \(F\) by multiplication by
\(
\prod_{v\in\Omega_k^f}
[\mathfrak p_v]^{\ord_v(\lambda)}
=
[(\lambda)]
=
1
\).
Thus the two subsets coincide.
\end{proof}

The relation between the local defects and the ideal class map is given
by the following formula.

\begin{prop}\label{prop:defect-formula}
For every \(P\in\P^n(k)\), we have
\[
\Phi_m(f(P))
=
\Phi_n(P)^d
\cdot
\prod_{v\in\Omega_k^f}
[\mathfrak p_v]^{\delta_{F,v}(P)}.
\]
\end{prop}

\begin{proof}
Choose \(x=(x_0,\ldots,x_n)\in k^{n+1}\setminus\{0\}\) representing \(P\).
For every \(v\in\Omega_k^f\), put \(a_v:=\min_j\ord_v(x_j)\) and \(b_v:=\min_i\ord_v(F_i(x))\).
Then \(\delta_{F,v}(P)=b_v-da_v\).

Moreover, we have
\[
x_0\Ok+\cdots+x_n\Ok
=
\prod_{v\in\Omega_k^f}\mathfrak p_v^{a_v}
\]
and
\[
F_0(x)\Ok+\cdots+F_m(x)\Ok
=
\prod_{v\in\Omega_k^f}\mathfrak p_v^{b_v}.
\]
Taking ideal classes gives
\[
\Phi_m(f(P))
=
\Phi_n(P)^d
\cdot
\prod_{v\in\Omega_k^f}[\mathfrak p_v]^{b_v-da_v}.
\]
Since \(b_v-da_v=\delta_{F,v}(P)\), the formula follows.
\end{proof}

We now determine the ideal class image of \(f\) on \(\Lambda\).

\begin{thm}\label{thm:morphism-image}
With the notation above, we have
\[
\Phi_m\bigl(f(\Lambda(k))\bigr)
=
\mathscr{D}_{f,\Lambda}\cdot\Cl(k)^d.
\]

In particular, the ideal class image of \(f(\Lambda(k))\) is a finite union of cosets of \(\Cl(k)^d\).
\end{thm}

\begin{proof}
Proposition~\ref{prop:defect-formula} immediately gives
\(
\Phi_m\bigl(f(\Lambda(k))\bigr)
\subset
\mathscr{D}_{f,\Lambda}\cdot\Cl(k)^d
\).

For the reverse inclusion, let $\Delta\in\mathscr{D}_{f,\Lambda}$ and \(c\in\Cl(k)\).
Choose \(a_v\in D_{F,v,\Lambda}\) for every
\(v\in\Omega_k^f\) such that
\(
\Delta=
\prod_{v\in\Omega_k^f}[\mathfrak p_v]^{a_v}
\).
Let
\[
T:=
\{v\in\Omega_k^f:D_{F,v,\Lambda}\neq\{0\}\}.
\]
This is a finite set by Lemma~\ref{lem:local-defect-properties}. For every \(v\in T\), choose
\(P_v\in\Lambda(k_v)\) such that
\(
\delta_{F,v}(P_v)=a_v
\).
Since \(\delta_{F,v}\) is locally constant, there is a non-empty open
subset
\(
\mathcal{U}_v\subset\Lambda(k_v)
\)
containing \(P_v\) on which \(\delta_{F,v}\) is identically equal to
\(a_v\).

By Theorem~\ref{thm:linear-subspaces}, there exists
\(P\in\Lambda(k)\) such that $\Phi_n(P)=c$ and \(P\in\mathcal{U}_v\) for every \(v\in T\).
It follows that
\(
\delta_{F,v}(P)=a_v
\)
for every \(v\in T\). If \(v\notin T\), then
\(D_{F,v,\Lambda}=\{0\}\), so
\(\delta_{F,v}(P)=0=a_v\). Proposition~\ref{prop:defect-formula}
therefore gives
\(
\Phi_m(f(P))
=
c^d\Delta
\).
This proves the reverse inclusion.
\end{proof}
\ 

The following example shows that the image in
Theorem~\ref{thm:morphism-image} can be a proper subset of \(\Cl(k)\)
that is a union of several cosets of a nontrivial subgroup; in
particular, the defect set \(\mathscr{D}_{f,\Lambda}\) is genuinely
needed.
 
\begin{exam}\label{exam:many-cosets}
Let \(k=\Q(\sqrt{-65})\), with \(\Ok=\Z[\sqrt{-65}]\) and norm form
\(N(a+b\sqrt{-65})=a^2+65b^2\). Its class group is isomorphic to \(\Z/2\times\Z/4\), as recorded in
\cite[\href{https://www.lmfdb.org/NumberField/2.0.260.1}{Number field 2.0.260.1}]{lmfdb}.

Since \(-65\equiv1\pmod3\), the prime \(3\) splits, say
\((3)=\mathfrak p_3\overline{\mathfrak p}_3\); let \(v_3\) and
\(\overline v_3\) be the corresponding finite places. The class
\(g:=[\mathfrak p_3]\) has order \(4\): the equation \(a^2+65b^2=3\) has
no solution, so \(\mathfrak p_3\) is not principal, and the only
elements of norm \(9\) are \(\pm3\), which generate
\((3)=\mathfrak p_3\overline{\mathfrak p}_3\neq\mathfrak p_3^{\,2}\), so
\(\mathfrak p_3^{\,2}\) is not principal either; as \(\Cl(k)\) has
exponent \(4\), the class \(g\) has order \(4\). Note also
\([\overline{\mathfrak p}_3]=g^{-1}=g^3\).

Consider the degree-two morphism
\[
f=[X^2:X^2+3Y^2]:\P^1_k\longrightarrow\P^1_k,
\qquad
F=(X^2,\,X^2+3Y^2),
\]
whose two forms have no common zero over \(\overline{k}\); take
\(\Lambda=\P^1_k\). The reductions \(\overline{F_0}=X^2\) and
\(\overline{F_1}=X^2+3Y^2\) have a common zero over the residue field of
a finite place \(v\) if and only if \(3\equiv0\) there, that is, exactly
for \(v\in\{v_3,\overline v_3\}\); by the proof of
Lemma~\ref{lem:local-defect-properties}, \(\delta_{F,v}\equiv0\) for
every other finite place \(v\).

Fix the place \(v=v_3\), so that \(\ord_v(3)=1\), and let
\(P=[x_0:x_1]\) with \((x_0,x_1)\in\Ov^2\) primitive. If
\(\ord_v(x_0)=0\), then \(\delta_{F,v}(P)=0\). If \(\ord_v(x_0)\geq1\),
then \(x_1\in\Ov^\times\) and
\[
\ord_v(3x_1^2)=1<2\ord_v(x_0)=\ord_v(x_0^2),
\]
so \(\min_i\ord_v(F_i(x))=\ord_v(x_0^2+3x_1^2)=1\), whence
\(\delta_{F,v}(P)=1\). Thus \(D_{F,v_3,\Lambda}=\{0,1\}\), and likewise
\(D_{F,\overline v_3,\Lambda}=\{0,1\}\), while \(D_{F,v,\Lambda}=\{0\}\)
for all other \(v\). Therefore
\[
\mathscr{D}_{f,\Lambda}
=\bigl\{[\mathfrak p_3]^{a}[\overline{\mathfrak p}_3]^{b}
:a,b\in\{0,1\}\bigr\}
=\{1,\,g,\,g^3\}.
\]

Since \(\Cl(k)\cong\Z/2\times\Z/4\), the group \(\Cl(k)^2\) of squares
has order \(2\); as \(g^2\) is a nontrivial square,
\(\Cl(k)^2=\{1,g^2\}\). Theorem~\ref{thm:morphism-image} now gives
\[
\Phi_1\bigl(f(\P^1(k))\bigr)
=\{1,g,g^3\}\cdot\{1,g^2\}
=\{1,g,g^2,g^3\}
=\langle g\rangle.
\]
Thus the image is the cyclic subgroup \(\langle g\rangle\cong\Z/4\): it
is a union of the two cosets \(\{1,g^2\}\) and \(\{g,g^3\}\) of the
nontrivial subgroup \(\Cl(k)^2\), and it is a proper subset of
\(\Cl(k)\), since the four classes outside \(\langle g\rangle\) are not
represented. This should be contrasted with
Proposition~\ref{prop:integral-extension} below, where a morphism
admitting an integral extension has image a single coset.
\end{exam}
\ 

We conclude with the case of morphisms admitting an integral extension.

\begin{prop}\label{prop:integral-extension}
Suppose that \(f:\P^n_k\to\P^m_k\), with \(n\geq1\), admits an integral extension \(\mathcal F:\P^n_{\Ok}\to\P^m_{\Ok}\).
Then there exists an invertible \(\Ok\)-module \(\mathfrak a\) such that
\[
\mathcal F^*\mathcal O_{\P^m_{\Ok}}(1)
\simeq
\mathcal O_{\P^n_{\Ok}}(d)\otimes\pi^*\mathfrak a,
\]
where \(\pi:\P^n_{\Ok}\to\Spec\Ok\) is the structural morphism.
Moreover, for every \(P\in\P^n(k)\), one has \(\Phi_m(f(P))=\Phi_n(P)^d[\mathfrak a]\).
It follows that the ideal class image of \(f(\P^n(k))\) is given by
\[
\Phi_m\bigl(f(\P^n(k))\bigr)
=
[\mathfrak a]\cdot\Cl(k)^d.
\]
\end{prop}

\begin{proof}
After base change to \(k\), the line bundle \(\mathcal F^*\mathcal O_{\P^m_{\Ok}}(1)\) becomes
\(
f^*\mathcal O_{\P^m_k}(1)
\simeq
\mathcal O_{\P^n_k}(d)
\).
The computation of the Picard group of projective space in \cite[Ch.~II, Exercise~7.9]{Hartshorne1977} gives \(\Pic(\P^n_{\Ok})\cong\Z\times\Pic(\Ok)\), where the first factor is generated by \(\mathcal O(1)\).
Thus there is an invertible \(\Ok\)-module \(\mathfrak a\) such that
\[
\mathcal F^*\mathcal O_{\P^m_{\Ok}}(1)
\simeq
\mathcal O_{\P^n_{\Ok}}(d)\otimes\pi^*\mathfrak a.
\]

Let \(P\in\P^n(k)\), and let \(\widetilde P:\Spec\Ok\to\P^n_{\Ok}\) be its unique extension.
Then \[\mathcal F\circ\widetilde P:\Spec\Ok\to\P^m_{\Ok}\] is the unique extension of \(f(P)\).
Therefore, by Lemma~\ref{lem:phi-pullback},
\[
\begin{aligned}
\Phi_m(f(P))
&=
\bigl[(\mathcal F\circ\widetilde P)^*\mathcal O_{\P^m_{\Ok}}(1)\bigr] \\
&=
\bigl[\widetilde P^*\mathcal F^*\mathcal O_{\P^m_{\Ok}}(1)\bigr] \\
&=
\bigl[\widetilde P^*(\mathcal O_{\P^n_{\Ok}}(d)\otimes\pi^*\mathfrak a)\bigr].
\end{aligned}
\]
Since \(\widetilde P^*\mathcal O_{\P^n_{\Ok}}(d)\simeq(\widetilde P^*\mathcal O_{\P^n_{\Ok}}(1))^{\otimes d}\) and \(\pi\circ\widetilde P=\operatorname{id}_{\Spec\Ok}\), Lemma~\ref{lem:phi-pullback} gives
\[
\Phi_m(f(P))
=
\Phi_n(P)^d\cdot[\mathfrak a].
\]

Finally, since \(n\geq1\), Theorem~\ref{thm:linear-subspaces} applied to \(\Lambda=\P^n_k\) shows that \(\Phi_n\colon\P^n(k)\to\Cl(k)\) is surjective, and we conclude that
\[
\Phi_m\bigl(f(\P^n(k))\bigr)
=
[\mathfrak a]\cdot\Cl(k)^d.
\]
\end{proof}

\begin{rem}
In the situation of Proposition~\ref{prop:integral-extension}, the same argument together with Theorem~\ref{thm:linear-subspaces} shows that
\(
\Phi_m\bigl(f(\Lambda(k))\bigr)
=
[\mathfrak a]\cdot\Cl(k)^d
\)
for every positive-dimensional \(k\)-linear subspace \(\Lambda\subset\P^n_k\).
\end{rem}
\medskip

\bibliographystyle{amsalpha}
\bibliography{ref}

\end{document}